\documentclass[11pt]{amsart}

\usepackage{amssymb,bm,mathrsfs,enumerate,mathtools}
\usepackage{hyperref}

\newtheorem{thm}{Theorem}[section]
\newtheorem{prop}[thm]{Proposition}
\newtheorem{lem}[thm]{Lemma}
\newtheorem{cor}[thm]{Corollary}
\theoremstyle{definition}
\newtheorem{defn}[thm]{Definition}
\newtheorem{notn}[thm]{Notation}
\newtheorem{assmp}[thm]{Assumption}

\renewcommand{\v}{\bm{v}}
\newcommand{\w}{\bm{w}}
\renewcommand{\u}{\bm{u}}
\newcommand{\hypergeometricseries}[5]{ {}_{#1} F_{#2} \biggl( \!\! \begin{array}{c} #3 \\ #4 \end{array} \!\!\! \biggm| #5 \biggr)}
\newcommand{\basichypergeometricseries}[5]{ {}_{#1} \phi_{#2} \biggl( \!\! \begin{array}{c} #3 \\ #4 \end{array} \!\!\!\biggm| #5 \biggr)}

\begin{document}

\title{The Erd\H{o}s--Ko--Rado basis for a Leonard system}

\author{Hajime Tanaka}
\address{Research Center for Pure and Applied Mathematics, Graduate School of Information Sciences, Tohoku University, 6-3-09 Aramaki-Aza-Aoba, Aoba-ku, Sendai 980-8579, Japan}
\email{htanaka@m.tohoku.ac.jp}


\thanks{Supported in part by JSPS Grant-in-Aid for Scientific Research No.~23740002.}

\begin{abstract}
We introduce and discuss an \emph{Erd\H{o}s--Ko--Rado basis} of the vector space underlying a Leonard system $\Phi \,{=}\, \bigl( A ; A^* ; \{ E_i \}_{ i=0 }^d ; \{ E_i^* \}_{ i=0 }^d \bigr)$ that satisfies a mild condition on the eigenvalues of $A$ and $A^*$.
We describe the transition matrices to/from other known bases, as well as the matrices representing $A$ and $A^*$ with respect to the new basis.
We also discuss how these results can be viewed as a generalization of the linear programming method used previously in the proofs of the ``Erd\H{o}s--Ko--Rado theorems'' for several classical families of $Q$-polynomial distance-regular graphs, including the original 1961 theorem of Erd\H{o}s, Ko, and Rado.
\end{abstract}

\maketitle

\section{Introduction}

\emph{Leonard systems} \cite{Terwilliger2001LAA} naturally arise in representation theory, combinatorics, and the theory of orthogonal polynomials (see e.g. \cite{Terwilliger2003JCAM,Terwilliger2006N}).
Hence they are receiving considerable attention.
Indeed, the use of the name ``Leonard system'' is motivated by a connection to a theorem of Leonard \cite{Leonard1982SIAM}, \cite[pp.~263--274]{BI1984B}, which involves the $q$-Racah polynomials \cite{AW1979SIAM} and some related polynomials of the Askey scheme \cite{KS1998R}.
Leonard systems also play a role in coding theory; see \cite{KLM2010P}.

Let $\Phi = \bigl( A ; A^* ; \{ E_i \}_{ i=0 }^d ; \{ E_i^* \}_{ i=0 }^d \bigr)$ be a Leonard system over a field $\mathbb{K}$, and $V$ the vector space underlying $\Phi$ (see Section \ref{sec: Leonard systems} for formal definitions).
Then $V = \bigoplus_{ i=0 }^d E_i^* V$ and $\dim E_i^* V = 1$ ($0 \leqslant i \leqslant d$).
We have a ``canonical'' (ordered) basis of $V$ associated with this direct sum decomposition, called a \emph{standard basis}.
There are 8 variations for the standard basis.
Next, let $U_{ \ell } = \bigl( \sum_{ i=0 }^{ \ell } E_i^* V \bigr) \cap \bigl( \sum_{ j=\ell }^d E_j V \bigr)$ ($0 \leqslant \ell \leqslant d$).
Then, again it follows that $V = \bigoplus_{ \ell=0 }^d U_{ \ell }$ and $\dim U_{ \ell } = 1$ ($0 \leqslant \ell \leqslant d$).
We have a ``canonical'' basis of $V$ associated with this \emph{split decomposition}, called a \emph{split basis}.
The split decomposition is crucial in the theory of Leonard systems,\footnote{In some cases, $V$ has the structure of an evaluation module of the quantum affine algebra $U_q ( \widehat{ \mathfrak{sl} }_2 )$, and the split decomposition corresponds to its weight space decomposition; see e.g. \cite{IT2009KJM}.} and there are 16 variations for the split basis.
Altogether, Terwilliger \cite{Terwilliger2002RMJM} defined $24$ bases of $V$ and studied in detail the transition matrices between these bases as well as the matrices representing $A$ and $A^*$ with respect to them.

In the present paper, we introduce another basis of $V$, which we call an \emph{Erd\H{o}s--Ko--Rado} (or \emph{EKR}) \emph{basis} of $V$, under a mild condition on the eigenvalues of $A$ and $A^*$ (see below).
As its name suggests, this basis arises in connection with the famous \emph{Erd\H{o}s--Ko--Rado theorem} \cite{EKR1961QJMO} in extremal set theory.
Indeed, Delsarte's \emph{linear programming method} \cite{Delsarte1973PRRS}, which is closely related to Lov\'{a}sz's $\vartheta$-function bound \cite{Lovasz1979IEEE,Schrijver1979IEEE} on the Shannon capacity of graphs, has been successfully used in the proofs of the ``Erd\H{o}s--Ko--Rado theorems'' for certain families of $Q$-\emph{polynomial distance-regular graphs}\footnote{$Q$-polynomial distance-regular graphs are thought of as finite/combinatorial analogues of compact symmetric spaces of rank one; see \cite[pp.~311--312]{BI1984B}.} \cite{Wilson1984C,FW1986JCTA,Tanaka2006JCTA,Tanaka2010pre} (including the original 1961 theorem of Erd\H{o}s et al.), and constructing appropriate feasible solutions to the dual programs amounts to describing the EKR bases for the Leonard systems associated with these graphs; see Section \ref{sec: applications}.
It seems that the previous constructions of the feasible solutions depend on the geometric/algebraic structures which are more or less specific to the family of graphs in question.
Our results give a uniform description of such feasible solutions in terms of the \emph{parameter arrays} of Leonard systems.

The contents of the paper are as follows.
Section \ref{sec: Leonard systems} reviews basic terminology, notation and facts concerning Leonard systems.
In Section \ref{sec: EKR basis}, we first study the subspaces $W_t = \bigl( E_0^* V + \sum_{ i=d-t+1 }^d E_i^* V \bigr) \cap \bigl( E_0 V + \sum_{ j=t+1 }^d E_j V \bigr)$ ($0 \leqslant t \leqslant d$).
We show that $\dim W_t = 1$ $( 0 \leqslant t \leqslant d )$, and that $V = \bigoplus_{ t=0 }^d W_t$ if and only if $q \ne -1$, or $q = -1$ and $d$ is even, where $q$ denotes a \emph{base} of $\Phi$ (which is determined by the recurrence satisfied by the eigenvalues of $A$ and $A^*$).
Assuming that this is the case, we then define an EKR basis associated with this direct sum decomposition.
We describe the transition matrices to/from $3$ bases out of the $24$ bases mentioned above ($2$ standard, $1$ split), as well as the matrices representing $A$ and $A^*$ with respect to the EKR basis.
Our main results are Theorems \ref{transition matrices to EKR basis}, \ref{transition matrices from EKR basis}, and \ref{A, A* in terms of EKR basis}.
Section \ref{sec: applications} is devoted to discussions of the connections and applications of these results to the Erd\H{o}s--Ko--Rado theorems.

\section{Leonard systems}\label{sec: Leonard systems}

Let $\mathbb{K}$ be a field, $d$ a positive integer, $\mathscr{A}$ a $\mathbb{K}$-algebra isomorphic to the full matrix algebra $\mathrm{Mat}_{ d+1 } ( \mathbb{K} )$, and $V$ an irreducible left $\mathscr{A}$-module.
We remark that $V$ is unique up to isomorphism, and that $V$ has dimension $d+1$.
An element $A$ of $\mathscr{A}$ is said to be \emph{multiplicity-free} if it has $d+1$ mutually distinct eigenvalues in $\mathbb{K}$.
Let $A$ be a multiplicity-free element of $\mathscr{A}$ and $\{ \theta_i \}_{ i=0 }^d$ an ordering of the eigenvalues of $A$.
Let $E_i : V \rightarrow V ( \theta_i )$ $( 0 \leqslant i \leqslant d )$ be the projection map onto $V ( \theta_i )$ with respect to $V = \bigoplus_{ i=0 }^d V ( \theta_i )$, where $V ( \theta_i ) = \{ \u \in V : A \u = \theta_i \u \}$.
We call $E_i$ the \emph{primitive idempotent} of $A$ associated with $\theta_i$.
Notice that the $E_i$ are polynomials in $A$.

A \emph{Leonard system} in $\mathscr{A}$ (\cite[Definition 1.4]{Terwilliger2001LAA}) is a sequence
\begin{equation}\label{Leonard system}
	\Phi = \bigl( A ; A^* ; \{ E_i \}_{ i=0 }^d ; \{ E_i^* \}_{ i=0 }^d \bigr)
\end{equation}
satisfying the following axioms (LS1)--(LS5):
\begin{enumerate}[(LS1)]
\item Each of $A,A^*$ is a multiplicity-free element in $\mathscr{A}$.\footnote{It is customary that $A^*$ denotes the conjugate transpose of $A$. It should be stressed that we are \emph{not} using this convention.}
\item $\{ E_i \}_{ i=0 }^d$ is an ordering of the primitive idempotents of $A$.
\item $\{ E_i^* \}_{ i=0 }^d$ is an ordering of the primitive idempotents of $A^*$.
\item $E_i^* A E_j^* = \begin{cases} 0 & \text{if } | i-j | > 1 \\ \ne 0 & \text{if } | i-j | = 1 \end{cases} \quad ( 0 \leqslant i , j \leqslant d )$.
\item $E_i A^* E_j = \begin{cases} 0 & \text{if } | i-j | > 1 \\ \ne 0 & \text{if } | i-j | = 1 \end{cases} \quad ( 0 \leqslant i , j \leqslant d )$.
\end{enumerate}
We say that $\Phi$ is \emph{over} $\mathbb{K}$.
We refer the reader to \cite{Terwilliger2001LAA,Terwilliger2004LAA,Terwilliger2006N} for background on Leonard systems.

Throughout the paper, $\Phi = \bigl( A ; A^* ; \{ E_i \}_{ i=0 }^d ; \{ E_i^* \}_{ i=0 }^d \bigr)$ shall always denote the Leonard system \eqref{Leonard system}.
Notice that the following are Leonard systems:
\begin{align*}
	\Phi^* &= \bigl( A^* ; A ; \{ E_i^* \}_{ i=0 }^d ; \{ E_i \}_{ i=0 }^d \bigr), \\
	\Phi^{ \downarrow } &= \bigl( A ; A^* ; \{ E_i \}_{ i=0 }^d ; \{ E_{ d-i }^* \}_{ i=0 }^d \bigr), \\
	\Phi^{ \Downarrow } &= \bigl( A ; A^* ; \{ E_{ d-i } \}_{ i=0 }^d ; \{ E_i^* \}_{ i=0 }^d \bigr).
\end{align*}
Viewing $*, \downarrow, \Downarrow$ as permutations on all Leonard systems,
\begin{equation*}
	*^2 = \downarrow^2 = \Downarrow^2 = 1, \quad \Downarrow * = * \downarrow, \quad \downarrow * = * \Downarrow, \quad \downarrow \Downarrow = \Downarrow \downarrow.
\end{equation*}
The group generated by the symbols $*, \downarrow, \Downarrow$ subject to the above relations is the dihedral group $D_4$ with $8$ elements.
We shall use the following notational convention:

\begin{notn}
For any $g \in D_4$ and for any object $f$ associated with $\Phi$, we let $f^g$ denote the corresponding object for $\Phi^{ g^{ -1 } }$; an example is $E_i^* ( \Phi ) = E_i ( \Phi^* )$.
\end{notn}

It is known (\cite[Theorem 6.1]{Terwilliger2004LAA}) that there is a unique antiautomorphism $\dag$ of $\mathscr{A}$ such that $A^{ \dag } = A$ and $A^{ * \dag } = A^*$.
From now on, let $\langle \cdot, \cdot \rangle : V \times V \rightarrow \mathbb{K}$ be a nondegenerate bilinear form on $V$ such that (\cite[Section 15]{Terwilliger2004LAA})
\begin{equation*}
	\langle X \u_1 , \u_2 \rangle = \langle \u_1 , X^{ \dag } \u_2 \rangle \quad ( \u_1 , \u_2 \in V, \ X \in \mathscr{A} ).
\end{equation*}
We shall write
\begin{equation*}
	|| \u ||^2 = \langle \u , \u \rangle \quad ( \u \in V).
\end{equation*}

\begin{notn}
Henceforth we fix a nonzero vector $\v^g$ in $E_0^g V$ for each $g \in D_4$.
We abbreviate $\v = \v^1$ where $1$ is the identity of $D_4$.
For convenience, we also assume $\v^{ g_1 } = \v^{ g_2 }$ whenever $E_0^{ g_1 } V = E_0^{ g_2 } V$ ($g_1, g_2 \in D_4$).
We remark that $|| \v^g ||^2$, $\langle \v^g , \v^{ * g } \rangle$ are nonzero for any $g \in D_4$; cf.~\cite[Lemma 15.5]{Terwilliger2004LAA}.
\end{notn}

We now recall a few direct sum decompositions of $V$, as well as (ordered) bases of $V$ associated with them.
First, $\dim E_i^* V = 1$ $( 0 \leqslant i \leqslant d )$ and $V = \bigoplus_{ i=0 }^d E_i^* V$.
By \cite[Lemma 10.2]{Terwilliger2004LAA}, $E_i^* \v \ne 0$ $( 0 \leqslant i \leqslant d )$, so that $\{ E_i^* \v \}_{ i=0 }^d$ is a basis of $V$, called a $\Phi$-\emph{standard basis} of $V$.
Next, let $U_{ \ell } = \bigl( \sum_{ i=0 }^{ \ell } E_i^* V \bigr) \cap \bigl( \sum_{ j=\ell }^d E_j V \bigr)$ ($0 \leqslant \ell \leqslant d$).
Then, again $\dim U_{ \ell } = 1$ $( 0 \leqslant \ell \leqslant d )$ and $V = \bigoplus_{ \ell=0 }^d U_{ \ell }$, which is referred to as the $\Phi$-\emph{split decomposition} of $V$ \cite{Terwilliger2006N}.
We observe $U_0 = E_0^* V$ and $U_d = E_d V$.
For $0 \leqslant i \leqslant d$, let $\theta_i$ be the eigenvalue of $A$ associated with $E_i$.
Then it follows that $( A - \theta_{ \ell } I ) U_{ \ell } = U_{ \ell+1 }$ and $( A^* - \theta_{ \ell }^* I ) U_{ \ell } = U_{ \ell-1 }$ for $0 \leqslant  \ell \leqslant d$, where $U_{ -1 } = U_{ d+1 } = 0$ \cite[Lemma 3.9]{Terwilliger2001LAA}.
For $0 \leqslant i \leqslant d$, let $\tau_i, \eta_i$ be the following polynomials in $\mathbb{K}[z]$:
\begin{equation*}
	\tau_i (z) = \prod_{ h=0 }^{ i-1 } ( z - \theta_h ), \quad \eta_i (z) = \tau_i^{ \Downarrow } (z) = \prod_{ h=0 }^{ i-1 } ( z - \theta_{ d-h } ).
\end{equation*}
From the above comments it follows that $\tau_{ \ell } ( A ) \v^* \in U_{ \ell }$ $( 0 \leqslant \ell \leqslant d )$ and $\{ \tau_{ \ell } ( A ) \v^* \}_{ \ell=0 }^d$ is a basis of $V$, called a $\Phi$-\emph{split basis} of $V$.
Moreover, there are nonzero scalars $\varphi_i$ $( 1 \leqslant i \leqslant d)$ in $\mathbb{K}$ such that $A^* \tau_{ \ell } ( A ) \v^* = \theta_{ \ell }^* \tau_{ \ell } ( A ) \v^* + \varphi_{ \ell } \tau_{ \ell-1 } ( A ) \v^*$ $( 1 \leqslant \ell \leqslant d)$.

Let $\phi_i = \varphi_i^{ \Downarrow }$ $( 1 \leqslant i \leqslant d )$.
The \emph{parameter array} of $\Phi$ is
\begin{equation*}
	p( \Phi ) = \bigl( \{ \theta_i \}_{ i=0 }^d ; \{ \theta_i^* \}_{ i=0 }^d ; \{ \varphi_i \}_{ i=1 }^d ; \{ \phi_i \}_{ i=1 }^d \bigr).
\end{equation*}
Terwilliger \cite[Theorem 1.9]{Terwilliger2001LAA} showed that the isomorphism class\footnote{A Leonard system $\Psi$ in a $\mathbb{K}$-algebra $\mathscr{B}$ is \emph{isomorphic} to $\Phi$ if there is a $\mathbb{K}$-algebra isomorphism $\gamma : \mathscr{A} \rightarrow \mathscr{B}$ such that $\Psi = \Phi^{ \gamma } := \bigl( A^{ \gamma } ; A^{ * \gamma } ; \{ E_i^{ \gamma } \}_{ i=0 }^d ; \{ E_i^{ * \gamma } \}_{ i=0 }^d \bigr)$.
} of $\Phi$ is determined by $p(\Phi)$ and gave a classification of the parameter arrays of Leonard systems; cf.~\cite[Section 5]{Terwilliger2005DCC}.
In particular, the sequences $\{ \theta_i \}_{ i=0 }^d$ and $\{ \theta_i^* \}_{ i=0 }^d$ are recurrent in the sense that there is a scalar $\beta \in \mathbb{K}$ such that
\begin{equation}\label{PA5}
	\frac{ \theta_{ i-2 } - \theta_{ i+1 } }{ \theta_{ i-1 } - \theta_i } = \dfrac{ \theta_{ i-2 }^* - \theta_{ i+1 }^* }{ \theta_{ i-1 }^* - \theta_i^* } = \beta+1 \quad ( 2 \leqslant i \leqslant d - 1 ).
\end{equation}
It also follows that
\begin{equation}\label{PA4}
	\phi_i = \varphi_1 \vartheta_i + ( \theta_i^* - \theta_0^* ) ( \theta_{ d-i+1 } - \theta_0 ) \quad ( 1 \leqslant i \leqslant d ),
\end{equation}
where
\begin{equation*}
	\vartheta_i = \sum_{ h=0 }^{ i-1 } \frac{ \theta_h - \theta_{ d-h } } { \theta_0 - \theta_d } \quad ( 1 \leqslant i \leqslant d ).
\end{equation*}
Notice that $\vartheta_1 = \vartheta_d = 1$.
Moreover,
\begin{equation}\label{properties of vartheta}
	\vartheta_{ d-i+1 } = \vartheta_i, \quad \vartheta_i^* = \vartheta_i \quad ( 1 \leqslant i \leqslant d ).
\end{equation}
The parameter array behaves nicely with respect to the $D_4$ action:
\begin{lem}[{\cite[Theorem~1.11]{Terwilliger2001LAA}}]\label{how D4 acts on parameter array}
The following hold.
\begin{align*}
	p( \Phi^* ) =& \bigl( \{ \theta_i^* \}_{ i=0 }^d ; \{ \theta_i \}_{ i=0 }^d ; \{ \varphi_i \}_{ i=1 }^d ; \{ \phi_{ d-i+1 } \}_{ i=1 }^d \bigr). \tag{i} \\
	p( \Phi^{ \downarrow } ) =& \bigl( \{ \theta_i \}_{ i=0 }^d ; \{ \theta_{ d-i }^* \}_{ i=0 }^d ; \{ \phi_{ d-i+1 } \}_{ i=1 }^d ; \{ \varphi_{ d-i+1 } \}_{ i=1 }^d \bigr). \tag{ii} \\
	p( \Phi^{ \Downarrow } ) =& \bigl( \{ \theta_{ d-i } \}_{ i=0 }^d ; \{ \theta_i^* \}_{ i=0 }^d ; \{ \phi_i \}_{ i=1 }^d ; \{ \varphi_i \}_{ i=1 }^d \bigr). \tag{iii}
\end{align*}
\end{lem}

The following can be easily read off \cite{Terwilliger2002RMJM,Terwilliger2004LAA}.
\begin{lem}[{\cite{Terwilliger2002RMJM,Terwilliger2004LAA}}]\label{transition matrices between split and standard bases}
The following hold.
\begin{align*}
	& E_i^* \v = \frac{ || E_i^* \v ||^2 }{ \langle \v , \v^* \rangle } \cdot \sum_{ \ell=0 }^i \frac{ \tau_{ \ell }^* ( \theta_i^* ) }{ \varphi_1 \dots \varphi_{ \ell } } \tau_{ \ell } ( A ) \v^* \quad ( 0 \leqslant i \leqslant d ). \tag{i} \\
	& \begin{aligned}[t] \tau_{ \ell } ( A ) \v^* =& \, \langle \v , \v^* \rangle \cdot \varphi_1 \dots \varphi_{ \ell } \\ & \times \sum_{ i=0 }^{ \ell } \frac{ \eta_{ d-\ell }^* ( \theta_i^* ) }{ \tau_i^* ( \theta_i^* ) \eta_{ d-i }^* ( \theta_i^* ) } \cdot \frac{ 1 }{ || E_i^* \v ||^2 } E_i^* \v \quad ( 0 \leqslant \ell \leqslant d ). \end{aligned} \tag{ii} \\
	& E_j \v^* = \sum_{ \ell=j }^d \frac{ \eta_{ d-\ell } ( \theta_j ) }{ \tau_j ( \theta_j ) \eta_{ d-j } ( \theta_j ) } \tau_{ \ell } ( A ) \v^* \quad ( 0 \leqslant j \leqslant d ). \tag{iii} \\
	& \tau_{ \ell } ( A ) \v^* = \sum_{ j=\ell }^d \tau_{ \ell } ( \theta_j ) E_j \v^* \quad ( 0 \leqslant \ell \leqslant d ). \tag{iv} \\
	& E_j \v^{ * \downarrow } = \dfrac{ \langle \v , \v^{ * \downarrow } \rangle }{ \langle \v , \v^* \rangle } \cdot \dfrac{ \phi_{ d-j+1 } \dots \phi_d }{ \varphi_1 \dots \varphi_j} E_j \v^* \quad ( 0 \leqslant j \leqslant d ). \tag{v}
\end{align*}
\end{lem}

Finally, it follows that (\cite[Lemma 9.2, Theorem 17.12]{Terwilliger2004LAA})
\begin{equation*}
	E_0^* E_i E_0^* = \frac{ \varphi_1 \dots \varphi_i \phi_1 \dots \phi_{ d-i } }{ \eta_d^* ( \theta_0^* ) \tau_i ( \theta_i ) \eta_{ d-i } (\theta_i ) } E_0^* \quad ( 0 \leqslant i \leqslant d ),
\end{equation*}
from which it follows that
\begin{equation}\label{squared norm}
	|| E_i^* \v ||^2 = \frac{ \varphi_1 \dots \varphi_i \phi_{ i+1 } \dots \phi_d }{ \eta_d ( \theta_0 ) \tau_i^* ( \theta_i^* ) \eta_{ d-i }^* (\theta_i^* ) } || \v ||^2 \quad ( 0 \leqslant i \leqslant d ),
\end{equation}
by virtue of Lemma \ref{how D4 acts on parameter array} (i).

\section{The Erd\H{o}s--Ko--Rado basis}\label{sec: EKR basis}

Let $F_{ \ell } : V \rightarrow U_{ \ell }$ $( 0 \leqslant \ell \leqslant d )$ be the projection map onto $U_{ \ell }$ with respect to the $\Phi$-split decomposition $V = \bigoplus_{ \ell=0 }^d U_{ \ell }$.

\begin{lem}[cf.~{\cite[Lemma 5.4]{ITT2001P}}]\label{when FE* and FE vanish}
The following hold.
\begin{align*}
	& F_{ \ell } E_i^* = 0 \ \text{if} \ \ell > i \quad ( 0 \leqslant i , \ell \leqslant d ). \tag{i} \\
	& F_{ \ell } E_j = 0 \ \text{if} \ \ell < j \quad ( 0 \leqslant j , \ell \leqslant d ). \tag{ii}
\end{align*}
\end{lem}

\begin{proof}
Immediate from $E_i^* V \subseteq \sum_{ \ell=0 }^i U_{ \ell }$ and $E_j V \subseteq \sum_{ \ell=j }^d U_{ \ell }$.
\end{proof}

\noindent
We shall mainly work with the $\Phi^{ \downarrow }$-split decomposition $V = \bigoplus_{ \ell=0 }^d U_{ \ell }^{ \downarrow }$, where
\begin{equation*}
	U_{ \ell }^{ \downarrow } = \Biggl( \sum_{ i=d-\ell }^d E_i^* V \Biggr) \cap \Biggl( \sum_{ j=\ell }^d E_j V \Biggr) \quad ( 0 \leqslant \ell \leqslant d ).
\end{equation*}
We now ``modify'' the $U_{ \ell }^{ \downarrow }$ and introduce the subspaces $W_t$ $( 0 \leqslant t \leqslant d )$ of $V$ defined by\footnote{The subscript $t$ is chosen in accordance with the concept of $t$-\emph{intersecting families} in the Erd\H{o}s--Ko--Rado theorem; see Section \ref{sec: applications}.}
\begin{equation*}
	W_t = \Biggl( E_0^* V + \sum_{ i=d-t+1 }^d E_i^* V \Biggr) \cap \Biggl( E_0 V + \sum_{ j=t+1 }^d E_j V \Biggr) \quad ( 0 \leqslant t \leqslant d ).
\end{equation*}
Observe $W_t \ne 0$ $( 0 \leqslant t \leqslant d )$, $W_0 = E_0^* V$, and $W_d = E_0 V$.
Notice also that
\begin{equation}\label{Phi*-EKR}
	W_t^* = W_{ d-t } \quad ( 0 \leqslant t \leqslant d ).
\end{equation}
Our aim is to show $\dim W_t = 1$ $( 0 \leqslant t \leqslant d )$, and then to determine precisely when $V = \bigoplus_{ t=0 }^d W_t$.
Pick $\w \in W_t$.
Then from Lemma \ref{when FE* and FE vanish} (applied to $\Phi^{ \downarrow }$) it follows that
\begin{equation*}
	F_{ \ell }^{ \downarrow } \w = \sum_{ i=0 }^{ d-\ell } F_{ \ell }^{ \downarrow } E_i^* \w = \sum_{ j=0 }^{ \ell } F_{ \ell }^{ \downarrow } E_j \w \quad ( 0 \leqslant \ell \leqslant d ).
\end{equation*}
Hence
\begin{equation}\label{Flw equals FlE0w or FlE0*w}
	F_{ \ell }^{ \downarrow } \w = \begin{cases} F_{\ell}^{\downarrow} E_0 \w & \text{if} \ 0 \leqslant \ell \leqslant t, \\ F_{\ell}^{\downarrow} E_0^* \w & \text{if} \ t \leqslant \ell \leqslant d, \end{cases}
\end{equation}
from which it follows that
\begin{equation}\label{Wt in terms of E0V, E0*V}
	\w = \sum_{ \ell=0 }^t F_{ \ell }^{ \downarrow } E_0 \w + \sum_{ \ell=t+1 }^d F_{ \ell }^{ \downarrow } E_0^* \w = E_0 \w + \sum_{ \ell=t+1 }^d F_{ \ell }^{ \downarrow } ( E_0^* - E_0 ) \w.
\end{equation}
By Lemma \ref{transition matrices between split and standard bases} (i) and Lemma \ref{how D4 acts on parameter array} (ii), we have
\begin{align}\label{E0*V part of Wt}
	F_{ \ell }^{ \downarrow } E_0^* \w =& F_{ \ell }^{ \downarrow } E_d^{ * \downarrow } \w \\
	=& \frac{ \langle \w, E_d^{ * \downarrow } \v^{ \downarrow } \rangle }{ || E_d^{ * \downarrow } \v^{ \downarrow } ||^2 } F_{ \ell }^{ \downarrow } E_d^{ * \downarrow } \v^{ \downarrow } \notag \\
	=& \frac{ \langle \w, E_d^{ * \downarrow } \v^{ \downarrow } \rangle }{ \langle \v^{ \downarrow } , \v^{ * \downarrow } \rangle } \cdot \frac{ \tau_{ \ell }^{ * \downarrow } ( \theta_d^{ * \downarrow } ) }{ \varphi_1^{ \downarrow } \dots \varphi_{ \ell }^{ \downarrow } } \tau_{ \ell }^{ \downarrow } ( A^{ \downarrow } ) \v^{ * \downarrow } \notag \\
	=& \frac{ \langle \w, E_0^* \v \rangle }{ \langle \v , \v^{ * \downarrow } \rangle } \cdot \frac{ \eta_{ \ell }^* ( \theta_0^* ) }{ \phi_{ d-\ell+1 } \dots \phi_d } \tau_{ \ell } ( A ) \v^{ * \downarrow } \notag
\end{align}
for $0 \leqslant \ell \leqslant d$.
Likewise, by Lemma \ref{transition matrices between split and standard bases} (iii) and Lemma \ref{how D4 acts on parameter array} (ii), we have
\begin{align}\label{E0V part of Wt}
	F_{ \ell }^{ \downarrow } E_0 \w =& F_{ \ell }^{ \downarrow } E_0^{ \downarrow } \w \\
	=& \frac{ \langle \w , E_0^{ \downarrow } \v^{ * \downarrow } \rangle }{ || E_0^{ \downarrow } \v^{ * \downarrow } ||^2 } F_{ \ell }^{ \downarrow } E_0^{ \downarrow } \v^{ * \downarrow } \notag \\
	=& \frac{ \langle \w , E_0 \v^{ * \downarrow } \rangle }{ || E_0 \v^{ * \downarrow } ||^2 } \cdot \frac{ \eta_{ d-\ell } ( \theta_0 ) }{ \eta_d ( \theta_0 ) } \tau_{ \ell } ( A ) \v^{ * \downarrow } \notag
\end{align}
for $0 \leqslant \ell \leqslant d$.
Since $F_t^{ \downarrow } E_0^* \w = F_t^{ \downarrow } E_0 \w$ by \eqref{Flw equals FlE0w or FlE0*w}, we have in particular:
\begin{equation}\label{essentially EKR bound}
	\frac{ \langle \w, E_0^* \v \rangle }{ \langle \v , \v^{ * \downarrow } \rangle } \cdot \frac{ \eta_t^* ( \theta_0^* ) }{ \phi_{ d-t+1 } \dots \phi_d } = \frac{ \langle \w, E_0 \v^{ * \downarrow } \rangle }{ || E_0 \v^{ * \downarrow } ||^2 } \cdot \frac{ \eta_{ d-t } ( \theta_0 ) }{ \eta_d ( \theta_0 ) }.
\end{equation}
Combining these comments, it follows from \eqref{Wt in terms of E0V, E0*V}, Lemma \ref{transition matrices between split and standard bases} (iv) and (v) that
\begin{align*}
	\w =& E_0 \w + \frac{ \langle \w, E_0 \v^{ * \downarrow } \rangle }{ || E_0 \v^{ * \downarrow } ||^2} \cdot \frac{ \eta_{ d-t } ( \theta_0 ) }{ \eta_d ( \theta_0 ) \eta_t^* ( \theta_0^* ) } \\
	& \times \sum_{ \ell=t+1 }^d \biggl( \frac{ \eta_{ \ell }^* ( \theta_0^* ) }{ \phi_{ d-\ell+1 } \dots \phi_{ d-t } } - \frac{ \eta_t^* ( \theta_0^* ) \eta_{ d-\ell } ( \theta_0 ) }{ \eta_{ d-t } ( \theta_0 ) } \biggr) \tau_{ \ell } ( A ) \v^{ * \downarrow } \\
	=& E_0 \w + \frac{ \langle \w, E_0 \v^* \rangle }{ || E_0 \v^* ||^2} \cdot \frac{ \eta_{ d-t } ( \theta_0 ) }{ \eta_d ( \theta_0 ) \eta_t^* ( \theta_0^* ) } \sum_{ j=t+1 }^d \frac{ \phi_{ d-j+1 } \dots \phi_d }{ \varphi_1 \dots \varphi_j } \notag \\
	& \times \sum_{ \ell=t+1 }^j \tau_{ \ell } ( \theta_j ) \biggl( \frac{ \eta_{ \ell }^* ( \theta_0^*) }{ \phi_{ d-\ell+1 } \dots \phi_{ d-t } } - \frac{ \eta_t^* ( \theta_0^* ) \eta_{ d-\ell } ( \theta_0 ) }{ \eta_{ d-t } ( \theta_0 ) } \biggr) E_j \v^*.
\end{align*}
The coefficient of the last sum is equal to $( \theta_j - \theta_0 )^{-1}$ times
\begin{align*}
	\sum_{ \ell = t + 1 }^j & ( \theta_j -\theta_{ \ell } + \theta_{ \ell } - \theta_0 ) \cdot \tau_{ \ell } ( \theta_j ) \biggl( \frac{ \eta_{ \ell }^* ( \theta_0^* ) }{ \phi_{ d-\ell+1 } \dots \phi_{ d-t } } - \frac{ \eta_t^* ( \theta_0^* ) \eta_{ d-\ell } ( \theta_0 ) }{ \eta_{ d-t } ( \theta_0 ) } \biggr) \\
	=& \sum_{ \ell = t+1 }^{ j-1 } \tau_{ \ell+1 } ( \theta_j ) \biggl( \frac{ \eta_{ \ell }^* ( \theta_0^* ) }{ \phi_{ d-\ell+1 } \dots \phi_{ d-t } } - \frac{ \eta_t^* ( \theta_0^* ) \eta_{ d-\ell } ( \theta_0 ) }{ \eta_{ d-t } ( \theta_0 ) } \biggr) \\
	& - \sum_{ \ell = t+1 }^j \tau_{ \ell } ( \theta_j ) \biggl( \frac{ \eta_{ \ell }^* ( \theta_0^* ) ( \theta_0 - \theta_{ \ell } ) }{ \phi_{ d-\ell+1 } \dots \phi_{ d-t } } - \frac{ \eta_t^* ( \theta_0^* ) \eta_{ d-\ell+1 } ( \theta_0 ) }{ \eta_{ d-t } ( \theta_0 ) } \biggr) \\
	=& \sum_{ \ell = t+1 }^j \tau_{ \ell } ( \theta_j ) \biggl( \frac{ \eta_{ \ell-1 }^* ( \theta_0^* ) }{ \phi_{ d-\ell+2 } \dots \phi_{ d-t } } - \frac{ \eta_{ \ell }^* ( \theta_0^* ) ( \theta_0 - \theta_{ \ell } ) }{ \phi_{ d-\ell+1 } \dots \phi_{ d-t } } \biggr) \\
	=& \sum_{ \ell = t+1 }^j \frac{ \tau_{ \ell } ( \theta_j ) \eta_{ \ell-1 }^* ( \theta_0^* ) }{ \phi_{ d-\ell+1 } \dots \phi_{ d-t } } \bigl( \phi_{ d-\ell+1 } - ( \theta_0^* - \theta_{ d-\ell+1 }^* ) ( \theta_0 - \theta_{ \ell } ) \bigr) \\
	=& \sum_{ \ell = t+1 }^j \frac{ \tau_{ \ell } ( \theta_j ) \eta_{ \ell-1 }^* ( \theta_0^* ) }{ \phi_{ d-\ell+1 } \dots \phi_{ d-t } } \varphi_1 \vartheta_{ \ell },
\end{align*}
where we have used \eqref{PA4} and \eqref{properties of vartheta}.
Hence
\begin{prop}\label{Wt in terms of standard}
Let $\w \in W_t$.
Then the following hold.
\begin{align*}
	\w =& E_0 \w + \frac{ \langle \w, E_0 \v^* \rangle }{ || E_0 \v^* ||^2} \cdot \frac{ \eta_{ d-t } ( \theta_0 ) }{ \eta_d ( \theta_0 ) \eta_t^* ( \theta_0^* ) } \tag{i} \\
	& \times \sum_{ j=t+1 }^d \frac{ \phi_{ d-j+1 } \dots \phi_d }{ \varphi_2 \dots \varphi_j ( \theta_j - \theta_0 ) } \Biggl( \sum_{ \ell = t+1 }^j \frac{ \tau_{ \ell } ( \theta_j ) \eta_{ \ell-1 }^* ( \theta_0^* ) \vartheta_{ \ell } }{ \phi_{ d-\ell+1 } \dots \phi_{ d-t } } \Biggr) E_j \v^*.
\end{align*}
\begin{align*}
	\w =& E_0^* \w + \frac{ \langle \w, E_0^* \v \rangle }{ || E_0^* \v ||^2} \cdot \frac{ \eta_t^* ( \theta_0^* ) }{ \eta_d^* ( \theta_0^* ) \eta_{ d-t } ( \theta_0 ) } \tag{ii} \\
	& \times \sum_{ i=d-t+1 }^d \frac{ \phi_1 \dots \phi_i }{ \varphi_2 \dots \varphi_i ( \theta_i^* - \theta_0^* ) } \Bigg( \sum_{ \ell = d-t+1 }^i \frac{ \tau_{ \ell }^* ( \theta_i^* ) \eta_{ \ell-1 } ( \theta_0 ) \vartheta_{ \ell } }{ \phi_{ d-t+1} \dots \phi_{ \ell } } \Biggr) E_i^* \v.
\end{align*}
In particular, $E_0 W_t \ne 0$, $E_0^* W_t \ne 0$, and $\dim W_t = 1$.
\end{prop}

\begin{proof}
(i):
Clear.

(ii):
By virtue of \eqref{Phi*-EKR}, the result follows from (i) above, together with Lemma \ref{how D4 acts on parameter array} (i) and \eqref{properties of vartheta}.

The last line follows by noting that each of $E_0 \w$, $E_0^* \w$ determines $\w$.
\end{proof}

\begin{notn}\label{beta}
Henceforth we let $q$ be a nonzero scalar in the algebraic closure $\overline{\mathbb{K}}$ of $\mathbb{K}$ such that $q + q^{ -1 } = \beta$, where the scalar $\beta$ is from \eqref{PA5}.
We call $q$ a \emph{base} for $\Phi$.\footnote{We may remark that if $d \geqslant 3$ then $\Phi$ has at most two bases, i.e., $q$ and $q^{ -1 }$.}
By convention, if $d<3$ then $q$ can be taken to be \emph{any} nonzero scalar in $\overline{\mathbb{K}}$.
\end{notn}

\begin{lem}[{cf.~\cite[(6.4)]{Tanaka2009LAAb}}]\label{when vartheta=0}
For $1 \leqslant i \leqslant d$, we have $\vartheta_i = 0$ precisely when $q = -1$, $d$ is odd, and $i$ is even.
\end{lem}

From Proposition \ref{Wt in terms of standard} and Lemma \ref{when vartheta=0}, it follows that
\begin{lem}\label{when diagonal entries are nonzero}
Let $q$ be as above.
Then for $1 \leqslant t \leqslant d-1$, the following hold. \\
\textup{(i)}
Suppose $q \ne -1$, or $q = -1$ and $d$ is even.
Then $E_{ d-t+1 }^* W_t \ne 0$ and $E_{ t+1 } W_t \ne 0$. \\
\textup{(ii)}
Suppose $q = -1$ and $d$ is odd.
Then $E_{ d-t+1 }^* W_t \ne 0$ (resp.~$E_{ t+1 } W_t \ne 0$) if and only if $t$ is odd (resp.~even).
\end{lem}

\begin{cor}\label{when EKR basis is defined}
Let $q$ be as above.
Then the following hold. \\
\textup{(i)}
Suppose $q \ne -1$, or $q = -1$ and $d$ is even.
Then $V = \bigoplus_{ t=0 }^d W_t$.
Moreover, $\sum_{ t=0 }^h W_t = E_0^* V + \sum_{ i=d-h+1 }^d E_i^* V$ and $\sum_{ t=h }^d W_t = E_0 V + \sum_{ j=h+1 }^d E_j V$ $( 0 \leqslant h \leqslant d )$. \\
\textup{(ii)}
Suppose $q = -1$ and $d$ is odd.
Then $W_{ 2 s-1 } = W_{ 2 s }$ for $1 \leqslant s \leqslant \lfloor d/2 \rfloor$.
\end{cor}

\begin{proof}
(i):
Immediate from Lemma \ref{when diagonal entries are nonzero} (i).

(ii):
It follows from Lemma \ref{when diagonal entries are nonzero} (ii) that
\begin{equation*}
	W_{ 2 s-1 } = \Biggl( E_0^* V + \sum_{ i=d-2 s+2 }^d E_i^* V \Biggr) \cap \Biggl( E_0 V + \sum_{ j=2 s+1 }^d E_j V \Biggr) = W_{ 2 s }
\end{equation*}
for $1 \leqslant s \leqslant \lfloor d/2 \rfloor$.
\end{proof}

By virtue of Corollary \ref{when EKR basis is defined}, we make the following assumption.
\begin{assmp}\label{EKR basis exists}
With reference to Notation \ref{beta}, for the rest of the paper we shall assume $q \ne -1$, or $q = -1$ and $d$ is even.\footnote{The Leonard systems with $d \geqslant 3$ that do \emph{not} satisfy this assumption are precisely those of \emph{Bannai/Ito} type \cite[Example 5.14]{Terwilliger2005DCC} with $d$ odd, and those of \emph{Orphan} type \cite[Example 5.15]{Terwilliger2005DCC}.}
\end{assmp}

We are now ready to introduce an Erd\H{o}s--Ko--Rado basis of $V$.

\begin{defn}
With reference to Assumption \ref{EKR basis exists}, for $0 \leqslant t \leqslant d$ let $\w_t$ be the (unique) vector in $W_t$ such that $E_0 \w_t = E_0 \v^*$.
We call $\{ \w_t \}_{ t=0 }^d$ a ($\Phi$-)\emph{Erd\H{o}s--Ko--Rado} (or \emph{EKR}) \emph{basis} of $V$.
\end{defn}

Notice that the basis $\{ \w_t \}_{ t=0 }^d$ linearly depends on the choice of $\v^* \in E_0^* V$.
In particular, we have $\w_0 = \v^*$ and $\w_d = E_0 \v^*$.
Our preference for the normalization $E_0 \w_t = E_0 \v^*$ comes from the applications to the Erd\H{o}s--Ko--Rado theorem; see Section \ref{sec: applications}.
The following theorem gives the transition matrix from each of the $\Phi^{ \downarrow }$-split basis $\{ \tau_{ \ell } ( A ) \v^{ * \downarrow } \}_{ \ell=0 }^d$, the $\Phi^*$-standard basis $\{ E_j \v^* \}_{ j=0 }^d$, and the $\Phi$-standard basis $\{ E_i^* \v \}_{ i=0 }^d$, to the EKR basis $\{ \w_t \}_{ t=0 }^d$.

\begin{thm}\label{transition matrices to EKR basis}
The following hold for $0 \leqslant t \leqslant d$.
\begin{align*}
	\w_t =& \frac{ \langle \v, \v^* \rangle }{ \langle \v, \v^{* \downarrow} \rangle } \Biggl\{ \sum_{ \ell=0 }^t \frac{ \eta_{ d-\ell } ( \theta_0 ) }{ \eta_d ( \theta_0 ) } \tau_{ \ell } ( A ) \v^{ * \downarrow } \tag{i} \\
	& + \frac{ \eta_{ d-t } ( \theta_0 ) }{ \eta_d ( \theta_0 ) \eta_t^* ( \theta_0^* ) } \sum_{ \ell=t+1 }^d \frac{ \eta_{ \ell }^* ( \theta_0^* ) }{ \phi_{ d-\ell+1 } \dots \phi_{ d-t } } \tau_{ \ell } ( A ) \v^{ * \downarrow } \Biggr\}.
\end{align*}
\begin{align*}
	\w_t =& E_0 \v^* + \frac{ \eta_{ d-t } ( \theta_0 ) }{ \eta_d ( \theta_0 ) \eta_t^* ( \theta_0^* ) } \tag{ii} \\
	& \times \sum_{ j=t+1 }^d \frac{ \phi_{ d-j+1 } \dots \phi_d }{ \varphi_2 \dots \varphi_j ( \theta_j - \theta_0 ) } \Biggl( \sum_{ \ell = t+1 }^j \frac{ \tau_{ \ell } ( \theta_j ) \eta_{ \ell-1 }^* ( \theta_0^* ) \vartheta_{ \ell } }{ \phi_{ d-\ell+1 } \dots \phi_{ d-t } } \Biggr) E_j \v^*.
\end{align*}
\begin{align*}
	\w_t =& \frac{ \langle \v , \v^* \rangle }{ || \v ||^2 } \Biggl\{ \frac{ \eta_d^* ( \theta_0^* ) \eta_{ d-t } ( \theta_0 ) }{ \phi_1 \dots \phi_{ d-t } \eta_t^* ( \theta_0^* ) } E_0^* \v \tag{iii} \\
	& + \sum_{ i=d-t+1 }^d \frac{ \phi_{ d-t+1 } \dots \phi_i }{ \varphi_2 \dots \varphi_i ( \theta_i^* - \theta_0^* ) } \Biggl( \sum_{ \ell = d-t+1 }^i \frac{ \tau_{ \ell }^* ( \theta_i^* ) \eta_{ \ell-1 } ( \theta_0 ) \vartheta_{ \ell } }{ \phi_{ d-t+1 } \dots \phi_{ \ell } } \Biggr) E_i^* \v \Biggr\}.
\end{align*}
\end{thm}

\begin{proof}
(i):
By Lemma \ref{transition matrices between split and standard bases} (v) and since $E_0 \w_t = E_0 \v^*$, we have
\begin{equation}\label{eq1}
	\frac{ \langle \w_t , E_0 \v^{* \downarrow} \rangle }{ || E_0 \v^{* \downarrow} ||^2 } = \frac{ \langle \w_t , E_0 \v^* \rangle }{ || E_0 \v^* ||^2 } \cdot \frac{ \langle \v , \v^* \rangle }{ \langle \v , \v^{ * \downarrow } \rangle } = \frac{ \langle \v , \v^* \rangle }{ \langle \v , \v^{ * \downarrow } \rangle }.
\end{equation}
Combining this with \eqref{essentially EKR bound}, it follows that
\begin{align}\label{eq2}
	E_0^* \w_t =& \frac{ \langle \w_t , E_0^* \v \rangle }{ || E_0^* \v ||^2 } E_0^* \v \\
	=& \frac{ \langle \v , \v^{* \downarrow} \rangle \langle \w_t , E_0 \v^{* \downarrow} \rangle }{ || E_0^* \v ||^2 || E_0 \v^{* \downarrow} ||^2 } \cdot \frac{ \phi_{ d-t+1 } \dots \phi_d \eta_{ d-t } ( \theta_0 ) }{ \eta_d ( \theta_0 ) \eta_t^* ( \theta_0^* ) } E_0^* \v \notag \\
	=& \frac{ \langle \v , \v^* \rangle }{ || E_0^* \v ||^2 } \cdot \frac{ \phi_{ d-t+1 } \dots \phi_d \eta_{ d-t } ( \theta_0 ) }{ \eta_d ( \theta_0 ) \eta_t^* ( \theta_0^* ) } E_0^* \v, \notag
\end{align}
from which it follows that
\begin{equation}\label{eq3}
	\frac{ \langle \w_t , E_0^* \v \rangle }{ \langle \v , \v^{ * \downarrow } \rangle } = \frac{ \langle \v , \v^* \rangle }{ \langle \v , \v^{ * \downarrow } \rangle } \cdot \frac{ \phi_{ d-t+1 } \dots \phi_d \eta_{ d-t } ( \theta_0 ) }{ \eta_d ( \theta_0 ) \eta_t^* ( \theta_0^* ) }.
\end{equation}
Now the result follows from \eqref{Wt in terms of E0V, E0*V}--\eqref{E0V part of Wt}, \eqref{eq1}, and \eqref{eq3}.

(ii):
Immediate from Proposition \ref{Wt in terms of standard} (i) and $E_0 \w_t = E_0 \v^*$.

(iii):
Follows from Proposition \ref{Wt in terms of standard} (ii), \eqref{squared norm}, and \eqref{eq2}.
\end{proof}

\begin{cor}\label{how * affect on EKR basis}
Let $\{ \w_t^* \}_{ t=0 }^d$ be the $\Phi^*$-EKR basis of $V$ normalized so that $E_0^* \w_t^* = E_0^* \v$ $( 0 \leqslant t \leqslant d )$.
Then
\begin{equation*}
	\w_t^* = \frac{ \langle \v , \v^* \rangle }{ || \v^* ||^2 } \cdot \frac{ \eta_d ( \theta_0 ) \eta_{ d-t }^* (\theta_0^*) }{ \phi_{ t+1 } \dots \phi_d \eta_t (\theta_0) } \w_{ d-t } \quad ( 0 \leqslant t \leqslant d ).
\end{equation*}
\end{cor}

\begin{proof}
By \eqref{Phi*-EKR}, $\w_t^*$ is a scalar multiple of $\w_{ d-t }$, and the scalar is found by looking at the coefficient of $E_0^* \v$ in $\w_{ d-t }$ as given in Theorem \ref{transition matrices to EKR basis} (iii), and by noting that $\langle \v , \v^* \rangle^2 || \v ^*||^{ -2 } = || E_0^* \v ||^2 = \phi_1 \dots \phi_d \eta_d ( \theta_0 )^{ -1 } \eta_d^* ( \theta_0^* )^{ -1 }  || \v ||^2$ in view of \eqref{squared norm}.
\end{proof}

Our next goal is to compute the transition matrix from the EKR basis $\{ \w_t \}_{t=0}^d$ to each of the three bases $\{ \tau_{\ell} (A) \v^{ * \downarrow } \}_{\ell=0}^d$, $\{ E_j \v^* \}_{j=0}^d$, and $\{ E_i^* \v \}_{i=0}^d$.
Let $G_t:V\rightarrow W_t$ $(0\leqslant t\leqslant d)$ be the projection map onto $W_t$ with respect to $V = \bigoplus_{t=0}^d W_t$.

\begin{lem}\label{when GE* and GE vanish}
The following hold.
\begin{align*}
	& G_t E_i^* = 0 \ \text{if} \ t>d-i+1, \ \text{or} \ t>0 \ \text{and} \ i=0 \quad ( 0 \leqslant i , t \leqslant d ). \tag{i} \\
	& G_t E_j = 0 \ \text{if} \ t<j-1, \ \text{or} \ t<d \ \text{and} \ j=0 \quad ( 0 \leqslant j , t \leqslant d ). \tag{ii}
\end{align*}
\end{lem}

\begin{proof}
Immediate from Corollary \ref{when EKR basis is defined} (i).
\end{proof}

For the moment, we write $\u = \u_{\ell} = \tau_{\ell} (A) \v^{ * \downarrow } \in U_{\ell}^{\downarrow}$.
Then it follows that
\begin{equation*}
	G_t \u = \sum_{i=d-\ell}^d G_t E_i^* \u = \sum_{j=\ell}^d G_t E_j \u \quad (0\leqslant t\leqslant d).
\end{equation*}
Hence it follows from Lemma \ref{when GE* and GE vanish} that
\begin{equation}\label{Gtu}
	G_t \u = \begin{cases} G_{\ell+1} E_{d-\ell}^* \u & \text{if} \ t = \ell+1, \\ G_{\ell} E_{\ell} \u + G_{\ell} E_{\ell+1} \u & \text{if} \ t = \ell, \\ G_{\ell-1} E_{\ell} \u & \text{if} \ t = \ell-1, \\ 0 & \text{if} \ t \leqslant \ell-2 \ \text{or} \ t \geqslant \ell+2. \end{cases}
\end{equation}
In particular:
\begin{equation}\label{ul has three terms}
	\u = G_{\ell-1} \u + G_{\ell} \u + G_{\ell+1} \u.
\end{equation}
By Lemma \ref{transition matrices between split and standard bases} (iv) and (v), we have
\begin{gather}
	E_{\ell} \u = \tau_{\ell} ( \theta_{\ell} ) E_{\ell} \v^{ * \downarrow } = \frac{ \langle \v , \v^{ * \downarrow } \rangle }{ \langle \v , \v^* \rangle } \cdot \frac{ \phi_{ d-\ell+1 } \dots \phi_d \tau_{\ell} ( \theta_{\ell} ) }{ \varphi_1 \dots \varphi_{\ell} } E_{\ell} \v^*, \label{Elu} \\
	E_{ \ell+1 } \u = \tau_{\ell} ( \theta_{ \ell+1 } ) E_{ \ell+1 } \v^{ * \downarrow } = \frac{ \langle \v , \v^{ * \downarrow } \rangle }{ \langle \v , \v^* \rangle } \cdot \frac{ \phi_{ d-\ell } \dots \phi_d \tau_{\ell} ( \theta_{ \ell+1 } ) }{ \varphi_1 \dots \varphi_{ \ell+1 } } E_{\ell+1} \v^*. \label{E(l+1)u}
\end{gather}
Likewise, by Lemma \ref{transition matrices between split and standard bases} (ii) and Lemma \ref{how D4 acts on parameter array} (ii),
\begin{align}
	E_{ d-\ell }^* \u =& E_{\ell}^{ * \downarrow } \u \label{E(d-l)*u} \\
	=& \langle \v^{ \downarrow } , \v^{ * \downarrow } \rangle \cdot \frac{ \varphi_1^{ \downarrow } \dots \varphi_{\ell}^{ \downarrow } }{ \tau_{\ell}^{ * \downarrow } ( \theta_{\ell}^{ * \downarrow } ) || E_{\ell}^{ * \downarrow } \v^{ \downarrow } ||^2 } E_{\ell}^{ * \downarrow } \v^{ \downarrow } \notag \\
	=& \langle \v , \v^{ * \downarrow } \rangle \cdot \frac{ \phi_{ d-\ell+1 } \dots \phi_d }{ \eta_{\ell}^* ( \theta_{ d-\ell }^* ) || E_{ d-\ell }^* \v ||^2 } E_{ d-\ell }^* \v. \notag
\end{align}
Notice that the transition matrix from the basis $E_1 \v^* , \dots, E_d \v^*, E_0 \v^*$ to the EKR basis $\w_0, \dots , \w_d$ is lower triangular.
Hence, for fixed $t$ with $0 \leqslant t \leqslant d-2$, if we write
\begin{align*}
	( E_{t+1} + E_{t+2} ) \w_t =& a E_{t+1} \v^* + b  E_{t+2} \v^*, \\
	( E_{t+1} + E_{t+2} ) \w_{t+1} =& c  E_{t+2} \v^*,
\end{align*}
then it follows that
\begin{align}
	( G_t + G_{t+1} ) E_{t+1} \v^* =& a^{-1} \w_t - a^{-1} c^{-1} b \w_{t+1}, \label{1/a, -b/ac} \\
	( G_t + G_{t+1} ) E_{t+2} \v^* =& c^{-1} \w_{t+1}.
\end{align}
By Theorem \ref{transition matrices to EKR basis} (ii), we routinely obtain
\begin{align}
	a^{-1} =& - \frac{ \varphi_2 \dots \varphi_{ t+1 } \eta_d ( \theta_0 ) }{ \phi_{ d-t+1 } \dots \phi_d \tau_{ t+1 } ( \theta_{ t+1 } ) \eta_{ d-t-1 } ( \theta_0 ) \vartheta_{ t+1 } }, \label{1/a} \\
	c^{-1} =& - \frac{ \varphi_2 \dots \varphi_{ t+2 } \eta_d ( \theta_0 ) }{ \phi_{ d-t } \dots \phi_d \tau_{ t+2 } ( \theta_{ t+2 } ) \eta_{ d-t-2 } ( \theta_0 ) \vartheta_{ t+2 } } , \\
	- a^{-1} c^{-1} b =& \frac{ \varphi_2 \dots \varphi_{ t+1 } \eta_d ( \theta_0 ) ( \theta_0 - \theta_{ t+1 } ) }{ \phi_{ d-t } \dots \phi_d \tau_{ t+1 } ( \theta_{ t+1 } ) \eta_{ d-t-1 } ( \theta_0 ) } \label{-b/ac} \\
	& \times \biggl( \frac{ \phi_{ d-t-1 } }{ ( \theta_{ t+2 } - \theta_{ t+1 } ) \vartheta_{ t+2 } } + \frac{ \theta_0^* - \theta_{ d-t }^* }{ \vartheta_{ t+1 } } \biggr). \notag
\end{align} 
From \eqref{Gtu}, \eqref{Elu}, \eqref{E(l+1)u}, and \eqref{1/a, -b/ac}--\eqref{-b/ac}, it follows that
\begin{align}
	G_{ \ell-1 } \u =& \frac{ \langle \v , \v^{ * \downarrow } \rangle }{ \langle \v , \v^* \rangle } \cdot \frac{ \phi_{ d-\ell+1 } \dots \phi_d \tau_{\ell} ( \theta_{\ell} ) }{ \varphi_1 \dots \varphi_{\ell} } G_{ \ell-1 } E_{\ell} \v^* \label{G(l-1)u} \\
	=& \frac{ \langle \v , \v^{ * \downarrow } \rangle }{ \langle \v , \v^* \rangle } \cdot \frac{ \phi_{ d-\ell+1 } \eta_d ( \theta_0 ) ( \theta_{ \ell } - \theta_0 ) }{ \varphi_1 \eta_{ d-\ell+1 } ( \theta_0 ) \vartheta_{ \ell } } \w_{ \ell-1 } \notag
\end{align}
when $1 \leqslant \ell \leqslant d$, and that
\begin{align}
	G_{\ell} \u =& \frac{ \langle \v , \v^{ * \downarrow } \rangle }{ \langle \v , \v^* \rangle } \biggl( \frac{ \phi_{ d-\ell+1 } \dots \phi_d \tau_{\ell} ( \theta_{\ell} ) }{ \varphi_1 \dots \varphi_{\ell} } G_{\ell} E_{\ell} \v^* \label{Glu} \\
	& + \frac{ \phi_{ d-\ell } \dots \phi_d \tau_{\ell} ( \theta_{ \ell+1 } ) }{ \varphi_1 \dots \varphi_{ \ell+1 } } G_{\ell} E_{\ell+1} \v^* \biggr) \notag \\
	=& \frac{ \langle \v , \v^{ * \downarrow } \rangle }{ \langle \v , \v^* \rangle } \cdot \frac{ \eta_d ( \theta_0 ) }{ \varphi_1 \eta_{ d-\ell } ( \theta_0 ) } \biggl( \frac{ \phi_{ d-\ell } }{ \vartheta_{ \ell+1 } } + \frac{ ( \theta_0 - \theta_{ \ell } ) ( \theta_0^* - \theta_{ d-\ell+1 }^* ) }{ \vartheta_{ \ell } } \biggr) \w_{\ell} \notag \\
	=& \frac{ \langle \v , \v^{ * \downarrow } \rangle }{ \langle \v , \v^* \rangle } \cdot \frac{ \eta_d ( \theta_0 ) }{ \varphi_1 \eta_{ d-\ell } ( \theta_0 ) } \biggl( \frac{ \phi_{ d-\ell } }{ \vartheta_{ \ell+1 } } + \frac{ \phi_{ d-\ell+1 } }{ \vartheta_{ \ell } } - \varphi_1 \biggr) \w_{\ell} \notag
\end{align}
when $1 \leqslant \ell \leqslant d-1$, where the last line follows from \eqref{PA4} and \eqref{properties of vartheta}.
When $\ell = 0$ or $\ell = d$, we interpret $\phi_0 / \vartheta_{d+1} = \phi_{d+1} / \vartheta_0 = \varphi_1$ in \eqref{Glu}.
Indeed, when $\ell = 0$, since $G_0 E_0 \u_0 = 0$ by Lemma \ref{when GE* and GE vanish} (ii), it follows from \eqref{Gtu}, \eqref{E(l+1)u}, \eqref{1/a, -b/ac}, and \eqref{1/a} that
\begin{equation*}
	G_0 \u_0 = G_0 E_1 \u_0 = \frac{ \langle \v , \v^{ * \downarrow } \rangle }{ \langle \v , \v^* \rangle } \cdot \frac{ \phi_d }{ \varphi_1 } G_0 E_1 \v^* = \frac{ \langle \v , \v^{ * \downarrow } \rangle }{ \langle \v , \v^* \rangle } \cdot \frac{ \phi_d }{ \varphi_1 } \w_0.
\end{equation*}
When $\ell = d$, since
\begin{align*}
	( E_d + E_0 ) \w_{d-1} =& - \frac{ \phi_2 \dots \phi_d \tau_d ( \theta_d ) }{ \varphi_2 \dots \varphi_d \eta_d ( \theta_0 ) } E_d \v^* + E_0 \v^*, \\
	( E_d + E_0 ) \w_d =& E_0 \v^*
\end{align*}
by Theorem \ref{transition matrices to EKR basis} (ii), it follows that
\begin{equation*}
	( G_{d-1} + G_d ) E_d \v^* = \frac{ \varphi_2 \dots \varphi_d \eta_d ( \theta_0 ) }{ \phi_2 \dots \phi_d \tau_d ( \theta_d ) } ( - \w_{d-1} + \w_d),
\end{equation*}
so that by \eqref{Gtu} and \eqref{Elu} we have
\begin{equation*}
	G_d \u_d = \frac{ \langle \v , \v^{ * \downarrow } \rangle }{ \langle \v , \v^* \rangle } \cdot \frac{ \phi_1 \dots \phi_d \tau_d ( \theta_d ) }{ \varphi_1 \dots \varphi_d } G_d E_d \v^* = \frac{ \langle \v , \v^{ * \downarrow } \rangle }{ \langle \v , \v^* \rangle } \cdot \frac{ \phi_1 \eta_d ( \theta_0 ) }{ \varphi_1 } \w_d.
\end{equation*}
Notice that the transition matrix from the basis $E_0^* \v , E_d^* \v , \dots , E_1^* \v$ to the EKR basis $\w_0 , \dots , \w_d$ is upper triangular.
Hence, for $1 \leqslant t \leqslant d$, since
\begin{equation*}
	E_{ d-t+1 }^* \w_t = \frac{ \langle \v , \v^* \rangle }{ || \v ||^2 } \cdot \frac{ \tau_{ d-t+1 }^* ( \theta_{ d-t+1 }^* ) \eta_{ d-t } (\theta_0) \vartheta_t }{ \varphi_2 \dots \varphi_{ d-t+1 } ( \theta_{ d-t+1 }^* - \theta_0^*) } E_{ d-t+1 }^* \v
\end{equation*}
by Theorem \ref{transition matrices to EKR basis} (iii) and \eqref{properties of vartheta}, it follows that
\begin{equation*}
	G_t E_{ d-t+1 }^* \v = \frac{ || \v ||^2 }{ \langle \v , \v^* \rangle } \cdot \frac{ \varphi_2 \dots \varphi_{ d-t+1 } ( \theta_{ d-t+1 }^* - \theta_0^*) }{ \tau_{ d-t+1 }^* ( \theta_{ d-t+1 }^* ) \eta_{ d-t } (\theta_0) \vartheta_t } \w_t,
\end{equation*}
so that by \eqref{Gtu}, \eqref{E(d-l)*u}, and \eqref{squared norm}, we have
\begin{align}
	G_{ \ell+1 } \u =& \langle \v , \v^{ * \downarrow } \rangle \cdot \frac{ \phi_{ d-\ell+1 } \dots \phi_d }{ \eta_{\ell}^* ( \theta_{ d-\ell }^* ) || E_{ d-\ell }^* \v ||^2 } G_{ \ell+1 } E_{ d-\ell }^* \v \label{G(l+1)u} \\
	=& \frac{ \langle \v , \v^{ * \downarrow } \rangle }{ \langle \v , \v^* \rangle } \cdot \frac{ \eta_d ( \theta_0 ) ( \theta_{ d-\ell }^* - \theta_0^*) }{ \varphi_1 \eta_{ d-\ell-1 } (\theta_0) \vartheta_{ \ell+1 } } \w_{ \ell+1 } \notag
\end{align}
when $0 \leqslant \ell \leqslant d-1$.

\begin{thm}\label{transition matrices from EKR basis}
Setting $\w_{ -1 } = \w_{ d+1 } = 0$, the following hold.\footnote{We also interpret the coefficients of $\w_{ -1 }$ and $\w_{ d+1 }$ as zero (or indeterminates), whenever these terms appear.}
\begin{align*}
	\tau_{\ell} (A) \v^{ * \downarrow } =& \frac{ \langle \v , \v^{ * \downarrow } \rangle }{ \langle \v , \v^* \rangle } \cdot \frac{ \eta_d ( \theta_0 ) }{ \varphi_1 } \biggl\{ - \frac{ \phi_{ d-\ell+1 } }{ \eta_{ d-\ell } ( \theta_0 ) \vartheta_{ \ell } } \w_{ \ell-1 } \tag{i} \\
	& + \frac{ 1 }{ \eta_{ d-\ell } ( \theta_0 ) } \biggl( \frac{ \phi_{ d-\ell } }{ \vartheta_{ \ell+1 } } + \frac{ \phi_{ d-\ell+1 } }{ \vartheta_{ \ell } } - \varphi_1 \biggr) \w_{\ell} \\
	& + \frac{ \theta_{ d-\ell }^* - \theta_0^* }{ \eta_{ d-\ell-1 } (\theta_0) \vartheta_{ \ell+1 } } \w_{ \ell+1 } \biggr\}
\end{align*}
for $0 \leqslant \ell \leqslant d$, where we interpret $\phi_0 / \vartheta_{d+1} = \phi_{d+1} / \vartheta_0 = \varphi_1$.
\begin{align*}
	E_j \v^* =& \frac{ \varphi_2 \dots \varphi_j \eta_d ( \theta_0 ) }{ \phi_{ d-j+1 } \dots \phi_d \tau_j ( \theta_j ) \eta_{ d-j } ( \theta_j ) } \biggl\{ - \frac{ \phi_{ d-j+1 } \eta_{ d-j } ( \theta_j ) }{ \eta_{ d-j } ( \theta_0 ) \vartheta_j } \w_{ j-1 } \tag{ii} \\
	&+ \begin{multlined}[t] ( \theta_j - \theta_0 ) \sum_{ t=j }^{d-1} \frac{ \eta_{ d-t-1 } ( \theta_j ) }{ \eta_{ d-t } ( \theta_0 )} \biggl( \frac{ \phi_{d-t} }{ \vartheta_{t+1} } \\
		+ \frac{ ( \theta_j - \theta_{t+1} ) ( \theta_{ d-t+1 }^* - \theta_0^* ) }{ \vartheta_t } \biggr) \w_t \end{multlined} \\
	&+ \bigl( \varphi_1 + ( \theta_1^* - \theta_0^* ) ( \theta_j - \theta_0 ) \bigr) \w_d \bigg\}
\end{align*}
for $1 \leqslant j \leqslant d$, and $E_0 \v^* = \w_d$.
\begin{align*}
	E_i^* \v =& \frac{ \langle \v , \v^* \rangle }{ || \v^* ||^2 } \cdot \frac{ \varphi_2 \dots \varphi_i \eta_d (\theta_0) \eta_d^* ( \theta_0^* ) }{ \phi_1 \dots \phi_i \tau_i^* ( \theta_i^* ) \eta_{ d-i }^* ( \theta_i^* ) } \biggl\{ \frac{ \varphi_1 + ( \theta_1 - \theta_0 ) ( \theta_i^* - \theta_0^* ) }{ \eta_d (\theta_0) } \w_0 \tag{iii} \\
	&+ \begin{multlined}[t] ( \theta_i^* - \theta_0^* ) \sum_{ t=1 }^{d-i} \frac{ \eta_{ t-1 }^* ( \theta_i^* ) }{ \phi_{ d-t+1 } \dots \phi_d \eta_{ d-t } (\theta_0) } \biggl( \frac{ \phi_{ d-t+1 } }{ \vartheta_t } \\
		+ \frac{ ( \theta_i^* - \theta_{ d-t+1 }^* ) ( \theta_{ t+1 } - \theta_0 ) }{ \vartheta_{ t+1 } } \biggr) \w_t \end{multlined} \\
	&+ \frac{ \eta_{ d-i }^* ( \theta_i^* ) ( \theta_i^* - \theta_0^* ) }{ \phi_{ i+1 } \dots \phi_d \eta_{ i-1 } (\theta_0) \vartheta_i } \w_{ d-i+1 } \biggr\}
\end{align*}
for $1 \leqslant i \leqslant d$, and $E_0^* \v = \langle \v , \v^* \rangle || \v^* ||^{ -2 } \w_0$.
\end{thm}

\begin{proof}
(i):
Immediate from \eqref{ul has three terms}, \eqref{G(l-1)u}, \eqref{Glu}, and \eqref{G(l+1)u}.

(ii):
By (i) above, Lemma \ref{transition matrices between split and standard bases} (iii) and (v), and Lemma \ref{how D4 acts on parameter array} (ii), we have
\begin{align*}
	E_j \v^* =& \frac{ \langle \v , \v^* \rangle }{ \langle \v , \v^{ * \downarrow } \rangle } \cdot \frac{ \varphi_1 \dots \varphi_j }{ \phi_{ d-j+1 } \dots \phi_d } \sum_{ \ell=j }^d \frac{ \eta_{ d-\ell } ( \theta_j ) }{ \tau_j ( \theta_j ) \eta_{ d-j } ( \theta_j ) } \tau_{ \ell } ( A ) \v^{ * \downarrow } \\
	=& \frac{ \varphi_2 \dots \varphi_j \eta_d ( \theta_0 ) }{ \phi_{ d-j+1 } \dots \phi_d \tau_j ( \theta_j ) \eta_{ d-j } ( \theta_j ) } \sum_{ \ell=j }^d \eta_{ d-\ell } ( \theta_j ) \biggl\{ \frac{ \phi_{ d-\ell+1 } ( \theta_{ \ell } - \theta_0 ) }{ \eta_{ d-\ell+1 } ( \theta_0 ) \vartheta_{ \ell } } \w_{ \ell-1 } \\
	& + \frac{ 1 }{ \eta_{ d-\ell } ( \theta_0 ) } \biggl( \frac{ \phi_{ d-\ell } }{ \vartheta_{ \ell+1 } } + \frac{ \phi_{ d-\ell+1 } }{ \vartheta_{ \ell } } - \varphi_1 \biggr) \w_{\ell} + \frac{ \theta_{ d-\ell }^* - \theta_0^* }{ \eta_{ d-\ell-1 } (\theta_0) \vartheta_{ \ell+1 } } \w_{ \ell+1 } \biggr\}
\end{align*}
for $1 \leqslant j \leqslant d$.
Now simplify the last line using \eqref{PA4} and \eqref{properties of vartheta}.

(iii):
Apply ``$*$'' to (ii) above with respect to the $\Phi^*$-EKR basis $\{ \w_t^* \}_{ t=0 }^d$ with $E_0^* \w_t^* = E_0^* \v$ $( 0 \leqslant t  \leqslant d )$, and then use Corollary \ref{how * affect on EKR basis}, Lemma \ref{how D4 acts on parameter array} (i), and \eqref{properties of vartheta}.
\end{proof}

Finally, we shall describe the matrices representing $A$ and $A^*$ with respect to the EKR basis $\{ \w_t \}_{ t=0 }^d$.
We use the following notation:
\begin{equation*}
	\Delta_s = \frac{ \eta_{ s-1 }^* ( \theta_0^* ) \bigl( ( \theta_{ d-s+1 }^* - \theta_0^* ) \vartheta_{ s+1 } - ( \theta_{ d-s }^* - \theta_0^* ) \vartheta_s \bigr) }{ \phi_{ d-s+1 } \dots \phi_d \eta_{ d-s-1 } ( \theta_0 ) \vartheta_{ s+1 } } \quad ( 1 \leqslant s \leqslant d-1 ).
\end{equation*}
Notice that
\begin{equation*}
	\Delta_s^* = \frac{ \eta_{ s-1 } ( \theta_0 ) \bigl( ( \theta_{ d-s+1 } - \theta_0 ) \vartheta_{ s+1 } - ( \theta_{ d-s } - \theta_0 ) \vartheta_s \bigr) }{ \phi_1 \dots \phi_s \eta_{ d-s-1 }^* ( \theta_0^* ) \vartheta_{ s+1 } } \quad ( 1 \leqslant s \leqslant d-1 ) ,
\end{equation*}
by virtue of  Theorem \ref{how D4 acts on parameter array} (i) and \eqref{properties of vartheta}.

\begin{thm}\label{A, A* in terms of EKR basis}
With the above notation, the following hold.
\begin{align*}
	A \w_t =& \theta_{t+1} \w_t + \biggl( \frac{ \phi_{ d-t+1 } \dots \phi_d \eta_{d-t} ( \theta_0 ) }{ \eta_t^* ( \theta_0^* ) } \Delta_{ t+1 } - ( \theta_{ t+1 } - \theta_0 ) \biggr) \w_{t+1} \tag{i} \\
	& + \frac{ \phi_{ d-t+1 } \dots \phi_d \eta_{d-t} ( \theta_0 ) }{ \eta_t^* ( \theta_0^* ) } \Biggl\{ \sum_{ s=t+2 }^{ d-1 } ( \Delta_s - \Delta_{ s-1 } ) \w_s - \Delta_{ d-1 } \w_d \Biggr\}
\end{align*}
for $ 0 \leqslant t \leqslant d-2$, $A \w_{ d-1 } = \theta_d \w_{ d-1 } - ( \theta_d - \theta_0 ) \w_d$, and $A \w_d = \theta_0 \w_d$.
\begin{align*}
	A^* \w_t =& - \frac{ \phi_1 \dots \phi_d }{ \eta_d ( \theta_0 ) } \Delta_{ d-1 }^* \w_0 \tag{ii} \\
	& + \sum_{ s=1}^{ t-2 } \frac{ \phi_1 \dots \phi_{ d-s } \eta_s^* ( \theta_0^* ) }{ \eta_{ d-s } ( \theta_0 ) } ( \Delta_{ d-s }^* - \Delta_{ d-s-1 }^* ) \w_s \\
	& + \biggl( \frac{ \phi_1 \dots \phi_{ d-t+1 } \eta_{ t-1 }^* ( \theta_0^* ) }{ \eta_{ d-t+1 } ( \theta_0 ) } \Delta_{ d-t+1 }^* - \frac{ \phi_{ d-t+1 } }{ \theta_t - \theta_0 } \biggr) \w_{t-1} + \theta_{ d-t+1 }^* \w_t
\end{align*}
for $2 \leqslant t \leqslant d$, $A^* \w_1 = \theta_d^* \w_1 - ( \theta_d^* - \theta_0^* ) \w_0$, and $A^* \w_0 = \theta_0^* \w_0$.
\end{thm}

\begin{proof}
(i)
By Theorem \ref{transition matrices to EKR basis} (i), \eqref{PA4}, \eqref{properties of vartheta}, and since $A \tau_{ \ell } (A) = \tau_{ \ell+1 } (A) + \theta_{ \ell } \tau_{ \ell } (A)$, we obtain
\begin{align*}
	A \w_t =& \frac{ \langle \v, \v^* \rangle }{ \langle \v, \v^{* \downarrow} \rangle } \Biggl\{ \sum_{\ell=1}^t \frac{ \eta_{d-\ell+1} (\theta_0) }{ \eta_d ( \theta_0 ) } \tau_{ \ell } (A) \v^{* \downarrow} + \sum_{\ell=0}^t \frac{ \eta_{d-\ell} (\theta_0) \theta_{ \ell } }{ \eta_d ( \theta_0 ) } \tau_{ \ell } ( A ) \v^{ * \downarrow } \\
	& + \frac{ \eta_{d-t} (\theta_0) }{ \eta_d (\theta_0) \eta_t^* (\theta_0^*) } \sum_{ \ell=t+1 }^d \frac{ \eta_{ \ell-1 }^* (\theta_0^*) }{ \phi_{ d-\ell+2 } \dots \phi_{d-t}} \tau_{ \ell } (A) \v^{* \downarrow} \\
	& + \frac{ \eta_{d-t} (\theta_0) }{ \eta_d (\theta_0) \eta_t^* (\theta_0^*) } \sum_{\ell=t+1}^d \frac{ \eta_{\ell}^* (\theta_0^*) \theta_{ \ell } }{ \phi_{d-\ell+1} \dots \phi_{d-t}} \tau_{ \ell } ( A ) \v^{ * \downarrow } \Biggr\} \\
	=& \frac{ \langle \v, \v^* \rangle }{ \langle \v, \v^{* \downarrow} \rangle } \Biggl\{ \theta_0 \sum_{ \ell=0 }^t \frac{ \eta_{d-\ell} (\theta_0) }{ \eta_d ( \theta_0 ) } \tau_{ \ell } ( A ) \v^{ * \downarrow } \\
	& + \frac{ \eta_{d-t} (\theta_0) \theta_0 }{ \eta_d (\theta_0) \eta_t^* (\theta_0^*) } \sum_{ \ell=t+1 }^d \frac{ \eta_{\ell}^* (\theta_0^*) }{ \phi_{d-\ell+1} \dots \phi_{d-t}} \tau_{ \ell } ( A ) \v^{ * \downarrow } \\
	& + \frac{ \varphi_1 \eta_{d-t} (\theta_0) }{ \eta_d (\theta_0) \eta_t^* (\theta_0^*) } \sum_{ \ell=t+1 }^d \frac{ \eta_{ \ell-1 }^* (\theta_0^*) \vartheta_{ d-\ell+1 } }{ \phi_{ d-\ell+1 } \dots \phi_{d-t}} \tau_{ \ell } (A) \v^{* \downarrow} \Biggr\} \\
	=& \theta_0 \w_t + \frac{ \langle \v, \v^* \rangle }{ \langle \v, \v^{* \downarrow} \rangle } \cdot \frac{ \varphi_1 \eta_{d-t} (\theta_0) }{ \eta_d (\theta_0) \eta_t^* (\theta_0^*) } \sum_{\ell=t+1}^d \frac{ \eta_{ \ell-1 }^* (\theta_0^*) \vartheta_{ \ell } }{ \phi_{ d-\ell+1 } \dots \phi_{d-t}} \tau_{ \ell } (A) \v^{* \downarrow}.
\end{align*}
Now apply Theorem \ref{transition matrices from EKR basis} (i) and simplify the result using \eqref{PA4} and \eqref{properties of vartheta}.

(ii):
Apply ``$*$'' to (i) above with respect to the $\Phi^*$-EKR basis $\{ \w_t^* \}_{ t=0 }^d$ such that $E_0^* \w_t^* = E_0^* \v$ $( 0 \leqslant t  \leqslant d )$, and then use Corollary \ref{how * affect on EKR basis}, Lemma \ref{how D4 acts on parameter array} (i), and \eqref{properties of vartheta}.
\end{proof}

We end this section with an attractive formula for $\Delta_s$.

\begin{lem}\label{unexpected formula}
For $1 \leqslant s \leqslant d-1$, we have
\begin{equation*}
	( \theta_{ d-s+1 } - \theta_0 ) \vartheta_{ s+1 } - ( \theta_{ d-s } - \theta_0 ) \vartheta_s =\frac{ \bigl( { \theta_{ d - \lfloor \frac{ s }{2} \rfloor } } - \theta_{ \lfloor \frac{ s }{2} \rfloor } \bigr) \bigl( \theta_{ d - \lfloor \frac{ s-1 }{2} \rfloor } - \theta_{ \lfloor \frac{ s+1 }{2} \rfloor } \bigr)  }{ \theta_d - \theta_0 }.
\end{equation*}
\end{lem}

\begin{proof}
This is verified case by case using \cite[Lemma 10.2]{Terwilliger2001LAA}.
\end{proof}

\begin{cor}
For $1 \leqslant s \leqslant d-1$, we have
\begin{equation*}
	\Delta_s = \frac{ \eta_{ s-1 }^* ( \theta_0^* ) \bigl( \theta_{ d - \lfloor \frac{ s }{2} \rfloor }^* - \theta_{ \lfloor \frac{ s }{2} \rfloor }^* \bigr) \bigl( \theta_{ d - \lfloor \frac{ s-1 }{2} \rfloor }^* - \theta_{ \lfloor \frac{ s+1 }{2} \rfloor }^* \bigr) }{ \phi_{d-s+1} \dots \phi_d \eta_{ d-s-1 } ( \theta_0 ) ( \theta_d^* - \theta_0^* ) \vartheta_{ s+1 } }.
\end{equation*}
\end{cor}

\begin{proof}
Immediate from Lemma \ref{unexpected formula} and \eqref{properties of vartheta}.
\end{proof}

\section{Applications to the Erd\H{o}s--Ko--Rado theorems}\label{sec: applications}

The Erd\H{o}s--Ko--Rado type theorems for various families of $Q$-polynomial distance-regular graphs provide one of the most successful applications of Delsarte's linear programming method \cite{Delsarte1973PRRS}.\footnote{See, e.g., \cite{DL1998IEEE,MT2009EJC} for more applications as well as extensions of this method.}

Let $\Gamma$ be a $Q$-polynomial distance-regular graph with vertex set $X$.
(We refer the reader to \cite{BI1984B,BCN1989B,Terwilliger1992JAC} for background material.)
Pick a ``base vertex'' $x \in X$ and let $\Phi = \Phi ( \Gamma )$ be the Leonard system (over $\mathbb{K} = \mathbb{R}$) afforded on the primary module of the Terwilliger algebra $\bm{T} ( x )$; cf.~\cite[Example (3.5)]{Tanaka2011EJC}.\footnote{We remark that $\Phi$ is independent of $x \in X$ up to isomorphism.}
The second eigenmatrix $Q = ( Q_{ij} )_{ i,j=0 }^d$ of $\Gamma$ is defined by\footnote{The matrix $Q$ is denoted $P^*$ in \cite[p.~264]{Terwilliger2004LAA}.}
\begin{equation*}
	E_j \v^* = \frac{ \langle \v, \v^* \rangle }{ || \v ||^2 } \sum_{i=0}^d Q_{ij} E_i^* \v \quad ( 0 \leqslant j \leqslant d ).
\end{equation*}
As summarized in \cite{Tanaka2010pre}, every ``$t$-intersecting family'' $Y \subseteq X$ is associated with a vector $\bm{e} = ( e_0 , \dots , e_d)$ (called the inner distribution of $Y$) satisfying
\begin{gather*}
	e_0=1 , \quad e_1 \geqslant 0 , \dots , e_{ d-t } \geqslant 0 , \quad e_{ d-t+1 } = \dots = e_d = 0 , \\
	| Y | = ( \bm{e} Q )_0, \quad ( \bm{e} Q )_1 \geqslant 0 , \dots , ( \bm{e} Q )_d \geqslant 0.
\end{gather*}
Viewing these as forming a linear programming maximization problem with objective function $( \bm{e} Q )_0$, we are then to construct a vector $\bm{f} = ( f_0 , \dots , f_d )$ such that
\begin{equation}\label{LP}
	f_0 = 1, \quad f_1 = \dots = f_t = 0, \quad ( \bm{f} Q^{ \mathsf{T} })_1 = \dots = ( \bm{f} Q^{ \mathsf{T} })_{d-t} = 0,
\end{equation}
which turns out to give a feasible solution to the dual problem with objective value $( \bm{f} Q^{ \mathsf{T} })_0$, provided that $f_{t+1} \geqslant 0 , \dots , f_d \geqslant 0$.

Set $\w = \sum_{j=0}^d f_j E_j \v^*$.
Then
\begin{equation*}
	\w = \frac{ \langle \v, \v^* \rangle }{ || \v ||^2 } \sum_{j=0}^d f_j \sum_{i=0}^d Q_{ij} E_i^* \v = \frac{ \langle \v, \v^* \rangle }{ || \v ||^2 } \sum_{i=0}^d ( \bm{f} Q^{ \mathsf{T} })_i E_i^* \v.
\end{equation*}
Hence it follows that $\bm{f}$ satisfies \eqref{LP} if and only if $\w = \w_t$.
In particular, such a vector $\bm{f}$ is unique and is given by Theorem \ref{transition matrices to EKR basis} (ii).

We now give three examples.
First, suppose $\Phi$ is of \emph{dual Hahn} type \cite[Example 5.12]{Terwilliger2005DCC}, i.e.,
\begin{equation*}
	\theta_i=\theta_0+hi(i+1+s), \quad \theta_i^*=\theta_0^*+s^*i
\end{equation*}
for $0\leqslant i\leqslant d$, and
\begin{align*}
	\varphi_i=hs^*i(i-d-1)(i+r), \quad \phi_i=hs^*i(i-d-1)(i+r-s-d-1)
\end{align*}
for $1\leqslant i\leqslant d$, where $h , s^*$ are nonzero.
Then it follows that
\begin{align*}
	f_j =& \frac{ ( 1-j )_t ( j+s+2 )_t ( s-r+1)_j ( -1 )^{ j-1 } }{ ( t-r+s+1 ) ( s+2 )_t t! ( r+2 )_{ j-1 } } \\
	& \times \hypergeometricseries{ 3 }{ 2 }{ t-j+1 , t+j+s+2 , 1 }{ t+1 , t-r+s+2 }{ 1 }
\end{align*}
for $t+1 \leqslant j \leqslant d$, and
\begin{equation*}
	( \bm{f} Q^{ \mathsf{T} } )_0 = \frac{ ( -d-s-1 )_{ d-t } }{ ( r-s-d )_{ d-t } }.
\end{equation*}
If $\Gamma$ is the Johnson graph $J ( v , d )$ \cite[Section 9.1]{BCN1989B}, then $\Phi$ is of dual Hahn type with $r = d-v-1$, $s = -v-2$, and $s^* = - v ( v-1 ) / d ( v-d )$; cf.~\cite[pp.~191--192]{Terwilliger1993JACb}.
In this case, the vector $\bm{f}$ was essentially constructed by Wilson \cite{Wilson1984C} and was used to prove the original Erd\H{o}s--Ko--Rado theorem \cite{EKR1961QJMO} in full generality.

Suppose $\Phi$ is of \emph{Krawtchouk} type \cite[Example 5.13]{Terwilliger2005DCC}, i.e.,
\begin{equation*}
	\theta_i = \theta_0+si, \quad \theta_i^* = \theta_0^* + s^* i
\end{equation*}
for $0\leqslant i\leqslant d$, and
\begin{equation*}
	\varphi_i = r i ( i-d-1 ), \quad \phi_i = ( r - ss^* ) i ( i-d-1 )
\end{equation*}
for $1\leqslant i\leqslant d$, where $r , s , s^*$ are nonzero.
Then it follows that
\begin{equation*}
	f_j = \frac{ ( 1-j )_t }{ t! } \biggl( \frac{ r - s s^* }{ r } \biggr)^{ j-1 } \cdot \hypergeometricseries{ 2 }{ 1 }{ t-j+1 , 1 }{ t+1 }{ \frac{ s s^* }{ s s^* - r } }
\end{equation*}
for $t+1 \leqslant j \leqslant d$, and
\begin{equation*}
	( \bm{f} Q^{ \mathsf{T} } )_0 = \biggl( \frac{ s s^* }{ s s^* - r } \biggr)^{ d-t }.
\end{equation*}
If $\Gamma$ is the Hamming graph $H ( d , n )$ \cite[Section 9.2]{BCN1989B}, then $\Phi$ is of Krawtchouk type with $r = n ( n-1 )$ and $s = s^* = -n$; cf.~\cite[p.~195]{Terwilliger1993JACb}.
In this case, the vector $\bm{f}$ coincides (up to normalization) with the weight distribution of an \emph{MDS code} \cite[Chapter 11]{MS1977B}, i.e., a code attaining the Singleton bound.\footnote{In this regard, one may also wish to call $\{ \w_t \}_{ t=0 }^d$ an \emph{MDS} basis or a \emph{Singleton} basis.}

Finally, suppose $\Phi$ is of the most general $q$-\emph{Racah} type \cite[Example 5.3]{Terwilliger2005DCC}, i.e.,
\begin{align*}
	\theta_i=\theta_0+h(1-q^i)(1-sq^{i+1})q^{-i}, \quad \theta_i^*=\theta_0^*+h^*(1-q^i)(1-s^*q^{i+1})q^{-i}
\end{align*}
for $0\leqslant i\leqslant d$, and
\begin{align*}
	\varphi_i&=hh^*q^{1-2i}(1-q^i)(1-q^{i-d-1})(1-r_1q^i)(1-r_2q^i), \\
	\phi_i&= hh^*q^{1-2i}(1-q^i)(1-q^{i-d-1})(r_1-s^*q^i)(r_2-s^*q^i)/s^*
\end{align*}
for $1\leqslant i\leqslant d$, where $h, h^*, r_1, r_2 , s , s^* , q$ are nonzero and $r_1r_2=ss^*q^{d+1}$.
Then it follows that the $f_j$ are expressed as balanced $_4\phi_3$ series:
\begin{align*}
	f_j =& \frac{ s^{* j-1 } q^{ ( d+1 ) ( j-1 ) + t } ( q^{ 1-j } ; q )_t ( s q^{ j+2 } ; q )_t ( s q / r_1 ; q )_j ( s q / r_2 ; q )_j }{ ( 1 - s q^{ t+1 } / r_1 ) ( 1 - s q^{ t+1 } / r_2 ) ( q ; q )_t ( s q^2 ; q )_t ( r_1 q^2 ; q )_{ j-1 } ( r_2 q^2 ; q )_{ j-1 } } \\
	& \times \basichypergeometricseries{ 4 }{ 3 }{ q^{ t-j+1 } , s q^{ t+j+2 } , q^{ t-d-1 } / s^* , q }{ q^{ t+1 } , s q^{ t+2 } / r_1 , s q^{ t+2 } / r_2 }{ q;q }
\end{align*}
for $t+1 \leqslant j \leqslant d$, and
\begin{equation*}
	( \bm{f} Q^{ \mathsf{T} } )_0 = \frac{ ( s q^{ t+2 } ; q )_{ d-t } ( s^* q^2 ; q )_{ d-t } }{ r_1^{ d-t } q^{ d-t } ( s q^{ t+1 } / r_1 ; q )_{ d-t } ( s^* q / r_1 ; q )_{ d-t } }.
\end{equation*}

\nocite{*}
\bibliographystyle{amsplain}


\end{document}